\documentclass[runningheads]{llncs}
\usepackage{libertine}

\usepackage{latexsym}
\usepackage{url,hyperref,bold-extra}
\usepackage{amssymb,amsfonts,amsmath}
\usepackage{pifont,yfonts,mathrsfs}

\usepackage{diagrams}

\begin{document}
\title{From Intuitionism to Many-Valued Logics through Kripke Models\thanks{Dedicated to Professor {\sc Mohammad Ardeshir} with high appreciation and admiration.}
}
%
%
\author{{\sc Saeed Salehi}
\orcidID{0000-0001-6961-5749}
}
\authorrunning{{\sc Saeed Salehi}}
%
\institute{Research Institute for Fundamental Sciences, University of Tabriz, 29 Bahman Boulevard, P.O.Box 51666-16471, Tabriz, {\sc Iran}. \\
School of Mathematics,
Institute for Research in Fundamental Sciences,  \\    P.O.Box~19395--5746,
Tehran, Iran.
\\
\email{\sf root@SaeedSalehi.ir \qquad www.SaeedSalehi.ir}
}
\maketitle              
\begin{abstract}
Intuitionistic Propositional Logic  is
proved to be an infinitely many valued logic by Kurt G\"odel~(1932), and it is proved by Stanis{\l}aw Ja\'{s}kowski~(1936) to be a countably 
many valued logic. In this paper, we provide alternative proofs for these theorems by using models of Saul Kripke~(1959).  G\"odel's proof gave rise to an intermediate propositional logic  (between intuitionistic and classical), that is known nowadays as G\"odel or the G\"odel-Dummet Logic, and is  studied by  fuzzy logicians as well. We also provide some results on the inter-definability of propositional connectives in this logic.

\end{abstract}
\section{Introduction  and Preliminaries}
Intuitionism grew out of some of the philosophical ideas of its founding father, Luitzen Egbertus Jan Brouwer (see e.g. \cite{brouwer}); what is known nowadays as intuitionistic logic is a formalization given by his student Arend Heyting \cite{heyting}. Kripke models (originating from \cite{kripke}) provided an interesting  mathematical interpretation for this formalization. Let us review some preliminaries about these models:

\begin{definition}[Kripke Frames]\label{def:frame}{\rm

\noindent
A {\em Kripke frame} is a partially ordered set; i.e.,
an ordered pair $\langle K,\succcurlyeq\rangle$ where
$\succcurlyeq\,\subseteq K^{2}$ is a reflexive, transitive and anti-symmetric binary relation on $K$.
}\hfill\ding{71} \end{definition}

%

\begin{definition}[Atoms, Formulas, Languages]\label{def:atfm}{\rm

\noindent
Let ${\tt At}$
be the set of all the propositional atoms; atoms are usually denoted by letters $\mathfrak{p}$ or $\mathfrak{q}$.
Let $\boldsymbol\top$ denote the verum (truth) constant.

\noindent
The language of propositional logics  studied here is $\mathscr{L}=\{\neg,\wedge,\vee,\rightarrow,
\boldsymbol\top\}$.

\noindent
For any $A\subseteq{\tt At}$ and $B\subseteq\mathscr{L}$, the set of all the formulas constructed from $A$ by means of $B$ is denoted by $\mathcal{L}(B,A)$.

\noindent
Let ${\tt Fm}$ denote the set of all the formulas; i.e., 
$\mathcal{L}(\mathscr{L},{\tt At})$.
}\hfill\ding{71} \end{definition}

\begin{definition}[Kripke Models]\label{def:model}{\rm

\noindent
A {\em Kripke model} is a triple $\mathcal{K}=\langle K,\succcurlyeq,\Vdash\rangle$, where $\langle K,\succcurlyeq\rangle$
is a Kripke frame equipped with a persistent binary (satisfaction) relation $\Vdash\,\subseteq K\times{\tt At}$; persistency (of the relation $\Vdash$ with respect to 
$\succcurlyeq$)  means that  for all $k,k'\in K$ and $\mathfrak{p}\in{\tt At}$, if $k'\!\succcurlyeq\!k\Vdash\mathfrak{p}$ then $k'\Vdash\mathfrak{p}$.

The satisfaction
relation  can be extended to all the (propositional) formulas,
i.e., to $\Vdash\;\subseteq K\times{\tt Fm}$, as follows:
\begin{itemize}
\item[$\circ$] $k\Vdash\boldsymbol\top$.
\item[$\circ$] $k\Vdash (\varphi\!\wedge\!\psi) \iff k\Vdash\varphi \textrm{ and }
k\Vdash\psi$.
\item[$\circ$] $k\Vdash (\varphi\!\vee\!\psi) \iff k\Vdash\varphi \textrm{ or }
k\Vdash\psi$.
\item[$\circ$]
    $k\Vdash(\neg\varphi) \iff  \forall k'\!\succcurlyeq\!k (k'\nVdash\varphi)$.
\item[$\circ$]
    $k\Vdash(\varphi\!\rightarrow\!\psi) \iff \forall k'\!\succcurlyeq\!k (k'\Vdash\varphi\Rightarrow  k'\Vdash\psi)$.
    \hfill\ding{71}
\end{itemize}
} \end{definition}

\begin{remark}[\textbf{On Persistency and its Converse}]\label{rem}{\rm

\noindent
It can be shown that the persistency conditions is inherited by the formulas; i.e., for any $k,k'\in K$ in any Kripke model $\mathcal{K}=\langle K,\succcurlyeq,\Vdash\rangle$ and for any formula $\varphi$,
if $k'\!\succcurlyeq\!k\Vdash\varphi$ then $k'\Vdash\varphi$.

\noindent
Obviously, the converse may not hold ($k'\Vdash\psi$ and $k'\!\succcurlyeq\!k$ do not necessarily imply that $k\Vdash\psi$); however, a partial converse holds for negated formulas:

if $k'\!\succcurlyeq\!k$ and $k'\Vdash\neg\varphi$, then $k\nVdash\varphi$.
}\hfill\ding{71}
\end{remark}

By the soundness and  completeness of the intuitionistic propositional logic  (IPL)  with respect to finite Kripke models, the tautologies of IPL are the  formulas (in ${\tt Fm}$) that are satisfied in all the elements of any finite Kripke model. A super-intuitionistic and sub-classical logic is the so-called G\"odel-Dummet logic (see \cite{dummet}), whose tautologies are the formulas that are satisfied in all the elements of all the connected finite Kripke models. A kind  of Kripke model theoretic characterization for this logic is given in \cite{ss-soco}.

\begin{definition}[Connectivity]\label{def:connected}
{\rm

\noindent
 A binary relation $R\subseteq K\times K$ is called {\em connected}, when for any $k,k',k'' \in K$, if $k'\!\succcurlyeq\!k$ and $k''\!\succcurlyeq\!k$, then  we have either $k'\!\succcurlyeq\!k''$  or  $k''\!\succcurlyeq\!k'$ (cf. \cite{svejd}).
}\hfill\ding{71} \end{definition}

The logic IPL is perhaps the most famous non-classical logic. A natural question (that according to Kurt G\"odel \cite{godel} was asked by his supervisor Hans Hahn) was whether IPL is a finitely many valued logic or not.
G\"odel \cite{godel} showed in 1932 that IPL is not finitely many valued. Stanis{\l}aw Ja\'skowski \cite{jaskowski} showed in 1936 that IPL is indeed a countably (infinite) many valued logic. In Section~\ref{sec:values} we give alternative proofs for these theorems by using Kripke models \cite{kripke} which were invented later in 1959.
G\"odel's proof gave birth to  an intermediate logic, that today is called the G\"odel-Dummet logic (GDL). Finally, in Section~\ref{sec:def} we study the problem of inter-definability of propositional connectives in GDL and IPL.

\section{$\omega-$Many Values for Intuitionistic Propositional Logic}\label{sec:values}
Let us begin with a formal definition of a many-valued logic. Throughout the paper, we are dealing with propositional logics only.

\begin{definition}[Many Valued Logics]\label{def:many}{\rm

\noindent
A {\em many valued logic} is $\langle\mathscr{V},\boldsymbol\tau,\backsim\!\!\!\!\!\backsim,\mathrm{\Lambda},\mathrm{V},
\boldsymbol=\!\!\!\boldsymbol>\rangle$, where $\mathscr{V}$ is  a set
of values with a {\em designated} element $\boldsymbol\tau\in\mathscr{V}$ (interpreted as the {\em  truth}) and the  functions
$\backsim\!\!\!\!\!\backsim\colon\mathscr{V}\!\rightarrow\!\mathscr{V}$,
$\mathrm{\Lambda}\colon\mathscr{V}^2\!\rightarrow\!\mathscr{V}$,
$\mathrm{V}\colon\mathscr{V}^2\!\rightarrow\!\mathscr{V}$, and
$\boldsymbol=\!\!\!\boldsymbol>\colon\mathscr{V}^2\!\rightarrow\!\mathscr{V}$ constitute     a  {\em truth table} on $\mathscr{V}$.

\noindent
A {\em valuation function} is any mapping
$\boldsymbol\nu\colon{\rm At}\rightarrow\mathscr{V}$, which can be extended to all the formulas, denoted also by
$\boldsymbol\nu\colon{\rm Fm}\rightarrow\mathscr{V}$, as follows:
\begin{itemize}
\item[$\circ$]
$\boldsymbol\nu(\neg\varphi) \,=\; \;
\backsim\!\!\!\!\!\backsim\boldsymbol\nu(\varphi)$.

\item[$\circ$]
$\boldsymbol\nu(\varphi\!\wedge\!\psi) \,=\; \;
\boldsymbol\nu(\varphi)\;\mathrm{\Lambda}\;
\boldsymbol\nu(\psi)$.

\item[$\circ$]
$\boldsymbol\nu(\varphi\!\vee\!\psi) \,=\; \;
\boldsymbol\nu(\varphi)\;\mathrm{V}\;
\boldsymbol\nu(\psi)$.

\item[$\circ$]
$\boldsymbol\nu(\varphi\!\rightarrow\!\psi) =\,
\boldsymbol\nu(\varphi) \boldsymbol=\!\!\!\boldsymbol>
\boldsymbol\nu(\psi)$.
\end{itemize}
A formula $\theta$ is called {\em tautology}, when it is mapped to the designated value under any valuation function; i.e., $\boldsymbol\nu(\theta)=\boldsymbol\tau$ for any valuation $\boldsymbol\nu$.
}\hfill\ding{71} \end{definition}

Theorem~\ref{th:godel32} appears in  \cite{safari} and \cite{ss-farsi}. In the following,
the disjunction operation ($\vee$) is assumed to be commutative and associative.

\begin{lemma}[A Tautology in $\mathbf{n}$-Valued Logics]\label{lem:n}

\noindent
For any $n>1$, the formula $\bigvee\hspace{-0.75em}\bigvee_{i<j\leqslant n}(\mathfrak{p}_i\!\rightarrow\!\mathfrak{p_j})$ is a tautology in any $n$-valued logic in which the formula $(\mathfrak{p}\!\rightarrow\!\mathfrak{p})\!\vee\!
\mathfrak{q}$ is a tautology.
\end{lemma}

\begin{proof}

\noindent

\noindent
In an $n$-valued logic, the $n+1$ atoms $\{\mathfrak{p}_0,\mathfrak{p}_1,\cdots,\mathfrak{p}_n\}$
can take $n$ values. So, under a valuation function,  there should exist some $i<j\leqslant n$ such that $\mathfrak{p}_i$ and $\mathfrak{p}_j$ take the same value, by the Pigeonhole Principle. Since $(\mathfrak{p}\!\rightarrow\!\mathfrak{p})\!\vee\!
\mathfrak{q}$ is a tautology, then the formula $\bigvee\hspace{-0.75em}\bigvee_{i<j\leqslant n}(\mathfrak{p}_i\!\rightarrow\!\mathfrak{p_j})$ should be mapped to the designated value by all the valuation functions.
\hfill \ding{113} \end{proof}

The lemma implies that the formula $(A\!\rightarrow\!B)
\vee(A\!\rightarrow\!C)
\vee(B\!\rightarrow\!C)$ is a tautology in the classical propositional logic; this formula is not a tautology in the intuitionistic (or even G\"odel-Dummet) propositional logic.

\begin{theorem}[G\"odel 1932: IPL Is Not Finitely Many Valued]\label{th:godel32}

\noindent
Intuitionistic propositional logic is not finitely many valued.
\end{theorem}
\begin{proof}

\noindent

\noindent
By Lemma~\ref{lem:n} it suffices to show that for any $n>1$, the formula $\bigvee\hspace{-0.75em}\bigvee_{i<j\leqslant n}(\mathfrak{p}_i\!\rightarrow\!\mathfrak{p_j})$ is not a tautology in IPL. Consider the Kripke model $\mathcal{K}=\langle K,\succcurlyeq,\Vdash\rangle$ with

$K=\{k,k_0,k_1,\cdots,k_{n-1}\}$,

$\succcurlyeq\,=\{(k_i,k)\mid i\!<\!n\}\cup\{(k_i,k_i)\mid i\!<\!n\}\cup\{(k,k)\}$, and

$\Vdash\,\,=\,\{(k_0,\mathfrak{p}_0),(k_1,\mathfrak{p}_1),
\cdots,(k_{n-1},\mathfrak{p}_{n-1})\}$.
\begin{diagram}
\hspace{3em}\bullet\;k_0[\![\mathfrak{p}_0]\!] &  & \hspace{3em} \bullet\;k_1 [\![\mathfrak{p}_1]\!] & \qquad \cdots & \hspace{3.75em}\bullet\;k_{n-1} [\![\mathfrak{p}_{n-1}]\!] \\
 & \luTo &\uTo & \ruTo &   \\
& & \hspace{1.75em}\bullet\;k[\![]\!] & &
\end{diagram}
\noindent
For any $i<n$ we have $k_i\Vdash\mathfrak{p}_i$,  and also $k_i\nVdash\mathfrak{p}_j$ for any $j>i$. So, $k_i\nVdash\mathfrak{p}_i\!\rightarrow\!\mathfrak{p}_j$ for any $i<j\leqslant n$; which implies that $k\nVdash\bigvee\hspace{-0.75em}\bigvee_{i<j\leqslant n}(\mathfrak{p}_i\!\rightarrow\!\mathfrak{p_j})$.
\hfill \ding{113} \end{proof}

The rest of this section is devoted to proving Ja\'skowski's result (Theorem~\ref{th:jask}) that IPL is a countably infinite many valued logic.

\begin{definition}[Monotone  Functions]\label{def:perf}{\rm

\noindent
 For a Kripke frame $(K,\succcurlyeq)$, a function $f\colon K\rightarrow\{0,1\}$ is called {\em monotone}, when for any $k,k'\in K$, if $k'\!\succcurlyeq\!k$, then $f(k')\!\geqslant\!f(k)$. We indicate the monotonicity of $f$  by writing  $f\colon (K,\succcurlyeq)\rightarrow\{0,1\}$.
}\hfill\ding{71} \end{definition}

\begin{example}[ $\textswab{f}_{\mathcal{K}}^\psi$\! ]
\label{ex:f}
{\rm

\noindent

\noindent
For any Kripke model $\mathcal{K}=(K,\succcurlyeq,\Vdash)$ and any formula $\psi$, the function $$\textswab{f}_{\mathcal{K}}^\psi\colon K\rightarrow\{0,1\}, \qquad \textswab{f}_\mathcal{K}^\psi(k)=\begin{cases} 1 & \text{ if } k\Vdash\psi \\ 0 & \text{ if } k\nVdash\psi\end{cases}$$
is monotone.
}\hfill\ding{71} \end{example}

\begin{definition}[$\backsim\!\!\!\!\!\backsim,  \mathrm{\Lambda}, \mathrm{V}$ and  $\boldsymbol=\!\!\!\boldsymbol>$]\label{def:conn}{\rm

\noindent
For a Kripke frame $(K,\succcurlyeq)$ and monotone  functions $f,g\colon (K,\succcurlyeq)\rightarrow\{0,1\}$, let

\medskip

$\backsim\!\!\!\!\!\backsim\!f\colon K\rightarrow\{0,1\}$ be defined by $(\backsim\!\!\!\!\!\backsim\!f)(k)=\begin{cases}1 & \text{ if } \forall k'\!\succcurlyeq\!k (f(k')\!=\!0) \\ 0 & \text{ if } \exists k'\!\succcurlyeq\!k (f(k')\!=\!1)\end{cases}$,

$f \,\mathrm{\Lambda}\, g\colon K\rightarrow\{0,1\}$ be defined by $(f \,\mathrm{\Lambda}\, g)(k)=\min\{f(k),g(k)\}$,

\medskip

$f \,\mathrm{V}\, g\colon K\rightarrow\{0,1\}$ be defined by $(f \,\mathrm{V}\, g)(k)=\max\{f(k),g(k)\}$,

\medskip

$f\!\boldsymbol=\!\!\!\boldsymbol>\!g\colon K\rightarrow\{0,1\}$ be defined by

\hfill $(f\!\boldsymbol=\!\!\!\boldsymbol>\!g)(k)=\begin{cases}1 & \text{ if } \forall k'\!\succcurlyeq\!k (f(k')\!=\!1\!\Rightarrow\!g(k')\!=\!1) \\ 0 & \text{ if } \exists k'\!\succcurlyeq\!k (f(k')\!=\!1\;\&\; g(k')\!=\!0)\end{cases}$,

\noindent
for all $k\in K$.
}\hfill\ding{71} \end{definition}

\begin{definition}[Constant  Functions]\label{def:cons}{\rm

\noindent
Let $\mathbf{1}_K\colon K\rightarrow\{0,1\}$ be  the constant $1$ function, i.e., $\mathbf{1}_K(k)=1$ for all $k\in K$; and let $\mathbf{0}_K\colon K\rightarrow\{0,1\}$ be the constant $0$ function: $\mathbf{0}_K(k)=0$ for all $k\in K$.
}\hfill\ding{71} \end{definition}

It is easy to see that the functions $\mathbf{1}_K$ and $\mathbf{0}_K$ obey the rules of the classical propositional logic with the operations $\backsim\!\!\!\!\!\backsim,  \mathrm{\Lambda}, \mathrm{V}$ and  $\boldsymbol=\!\!\!\boldsymbol>$. For example, $(\backsim\!\!\!\!\!\backsim\!\mathbf{1}_K)=\mathbf{0}_K$, $(\mathbf{1}_K\,\mathrm{\Lambda}\,\mathbf{1}_K)=\mathbf{1}_K$, $(\mathbf{0}_K\,\mathrm{V}\,\mathbf{1}_K)=\mathbf{1}_K$ and $(\mathbf{1}_K\!\boldsymbol=\!\!\!\boldsymbol>\!\mathbf{0}_K)
=\mathbf{0}_K$.
We omit the proof of the following straightforward observation.

\begin{lemma}[Monotonicity of $\mathbf{1}_K,\mathbf{0}_K,\backsim\!\!\!\!\!\backsim\!f,f \,\mathrm{\Lambda}\, g,f \,\mathrm{V}\, g$ and $f\!\boldsymbol=\!\!\!\boldsymbol>\!g$]\label{lem:per}

\noindent
For any Kripke frame $(K,\succcurlyeq)$, the constant functions $\mathbf{1}_K$ and $\mathbf{0}_K$ are monotone, and if $f,g\colon (K,\succcurlyeq)\rightarrow\{0,1\}$ are monotone, then so are $\backsim\!\!\!\!\!\backsim\!f,f \,\mathrm{\Lambda}\, g,f \,\mathrm{V}\, g$ and $f\!\boldsymbol=\!\!\!\boldsymbol>\!g$.
\hfill \ding{113} \end{lemma}

Finally, we can provide the following  countably many values for IPL:

\begin{definition}[Countably Many Values for IPL]\label{def:values}{\rm

\noindent
Enumerate all the finite Kripke frames as $(K_0,\succcurlyeq_0),(K_1,\succcurlyeq_1),
(K_2,\succcurlyeq_2),\cdots$, where $K_n\subset\mathbb{N}$ for all $n\in\mathbb{N}$. Let

\noindent $\mathscr{V}=\{\langle f_0,f_1,f_2,\cdots\rangle\mid \forall n [f_n\colon (K_n,\succcurlyeq_n)\!\rightarrow\!\{0,1\}]\;\&\;$

\hfill $\exists N\!\in\!\mathbb{N}[(\forall n\!\geqslant\!N f_n\!=\!\mathbf{1}_{K_n})\text{ or }(\forall n\!\geqslant\!N f_n\!=\!\mathbf{0}_{K_n})]\}.$

\noindent
In the other words, the set of values $\mathscr{V}$ consists of all the sequences $\langle f_0,f_1,f_2,\cdots\rangle$ such that for each $n$, $f_n$ is a monotone function on $(K_n,\succcurlyeq_n)$, and the sequences are ultimately constant (from a step onward,  $f_n$'s are either all $\mathbf{1}_{K_n}$ or all $\mathbf{0}_{K_n}$).

\noindent
Let $\boldsymbol\tau=\langle\mathbf{1}_{K_0},
\mathbf{1}_{K_1},\mathbf{1}_{K_2},\cdots\rangle$ be the designated element (for truth).

\noindent
For $\mathfrak{f}=\langle f_0,f_1,f_2,\cdots\rangle\in\mathscr{V}$ and $\mathfrak{g}=\langle g_0,g_1,g_2,\cdots\rangle\in\mathscr{V}$, let (cf. Definition~\ref{def:conn})

$\backsim\!\!\!\!\!\backsim\mathfrak{f}=\langle \backsim\!\!\!\!\!\backsim\!f_0,\backsim\!\!\!\!\!\backsim\!f_1,\backsim\!\!\!\!\!\backsim\!f_2,\cdots\rangle$,

$\mathfrak{f}\,\mathrm{\Lambda}\,\mathfrak{g}=\langle f_0 \,\mathrm{\Lambda}\, g_0,f_1 \,\mathrm{\Lambda}\, g_1,f_2 \,\mathrm{\Lambda}\, g_2,\cdots\rangle$,

$\mathfrak{f}\,\mathrm{V}\,\mathfrak{g}=\langle f_0 \,\mathrm{V}\, g_0,f_1 \,\mathrm{V}\, g_1,f_2 \,\mathrm{V}\, g_2,\cdots\rangle$, and

$\mathfrak{f}\!\boldsymbol=\!\!\!\boldsymbol>\!\mathfrak{g}=\langle f_0\!\boldsymbol=\!\!\!\boldsymbol>\!g_0,f_1\!\boldsymbol=\!\!\!\boldsymbol>\!g_1,
f_2\!\boldsymbol=\!\!\!\boldsymbol>\!g_2,\cdots\rangle$.
}\hfill\ding{71} \end{definition}

It can be immediately seen that $\mathscr{V}$ is a countable set, and Lemma~\ref{lem:per} implies that $\mathscr{V}$ is closed under the operations $\backsim\!\!\!\!\!\backsim,\mathrm{\Lambda},\mathrm{V}$ and $\boldsymbol=\!\!\!\boldsymbol>$. Before proving the main theorem, we make a further definition and prove  an auxiliary lemma.

\begin{definition}[$\langle\!\langle\alpha\rangle\!\rangle_n$, $\boldsymbol\Vdash_n^{\boldsymbol\nu}$ and $\boldsymbol\nu_m^{\Vdash}$]\label{def:forv}{\rm

\noindent
For a sequence $\alpha$, let $\langle\!\langle\alpha\rangle\!\rangle_n$ denote its $n$-th element (if any), for any $n\in\mathbb{N}$.

\noindent
(1) Let a valuation $\boldsymbol\nu\colon{\tt At}\rightarrow\mathscr{V}$ be given. The satisfaction relation $\boldsymbol\Vdash_n^{\boldsymbol\nu}$ is defined on any finite Kripke frame $(K_n,\succcurlyeq_n)$, with $K_n\subset\mathbb{N}$ (see Definition~\ref{def:values}), by the following for any atom $\mathfrak{p}\in{\tt At}$ and any $k\in K_n$: \;  $k\boldsymbol\Vdash_n^{\boldsymbol\nu}\mathfrak{p} \iff \langle\!\langle\boldsymbol\nu(\mathfrak{p})
\rangle\!\rangle_n(k)=1$.

\noindent
(2) Let a Kripke model $\mathcal{K}=(K_m,\succcurlyeq_m,\Vdash)$ on the Kripke frame $(K_m,\succcurlyeq_m)$  be given (see  Definition~\ref{def:values}). Define the valuation $\boldsymbol\nu_m^{\Vdash}$ by  $$\boldsymbol\nu_m^{\Vdash}(\mathfrak{p})=
\langle\mathbf{1}_{K_0},\cdots,
\mathbf{1}_{K_{m-1}},
\textswab{f}_{\mathcal{K}}^{\mathfrak{p}},
\mathbf{1}_{K_{m+1}},\cdots\rangle$$
for any  $\mathfrak{p}\in{\tt At}$, where $\textswab{f}_{\mathcal{K}}^{\mathfrak{p}}\colon K_m\rightarrow\{0,1\}$ is the function that was defined   in Example~\ref{ex:f}:
$\textswab{f}_\mathcal{K}^{\mathfrak{p}}(k)=1$ if $k\Vdash\mathfrak{p}$, and $\textswab{f}_\mathcal{K}^{\mathfrak{p}}(k)=0$ if $k\nVdash\mathfrak{p}$, for any $k\in K_m$.
}\hfill\ding{71} \end{definition}

It is clear that the relation $\boldsymbol\Vdash_n^{\boldsymbol\nu}\;\subseteq K_n\times{\tt At}$ is  persistent.

\begin{lemma}[On  $\boldsymbol\Vdash_n^{\boldsymbol\nu}$ and $\boldsymbol\nu_m^{\Vdash}$]\label{lem:aux}

\noindent
{\em (1)} Let a valuation $\boldsymbol\nu\colon{\tt At}\rightarrow\mathscr{V}$ be given, and the satisfaction  relation $\Vdash_n^{\boldsymbol\nu}$ be defined on  $(K_n,\succcurlyeq_n)$ as in Definition~\ref{def:forv}. Then for any formula $\varphi\in{\tt Fm}$
 and any $k\in K_n$, we have   $k\boldsymbol\Vdash_n^{\boldsymbol\nu}\varphi \iff \langle\!\langle\boldsymbol\nu(\varphi)
\rangle\!\rangle_n(k)=1$.

\noindent
{\em (2)}
 Let a Kripke model $\mathcal{K}=(K_m,\succcurlyeq_m,\Vdash)$ be given on the  frame $(K_m,\succcurlyeq_m)$, and  the valuation $\boldsymbol\nu_m^{\Vdash}$
 be defined as in Definition~\ref{def:forv}. Then for any formula $\varphi\in{\tt Fm}$ and any $k\in K_m$, we have $k\nVdash\varphi
  \iff \langle\!\langle\boldsymbol\nu_m^{\Vdash}(\varphi)
 \rangle\!\rangle_m(k)=0$.
 \end{lemma}
\begin{proof}

\noindent

\noindent
Both assertions can be proved by induction on $\varphi$. They are clear for $\varphi=\boldsymbol\top$ and hold for atomic $\varphi\in{\tt At}$ by Definition~\ref{def:forv}. The inductive cases follow immediately from
Definitions~\ref{def:model}, \ref{def:many}, \ref{def:conn}, and~\ref{def:values}.
\hfill \ding{113} \end{proof}

\begin{theorem}[Ja\'skowski 1936: IPL Is Countably  Many Valued]\label{th:jask}

\noindent
Intuitionistic propositional logic is countably infinite  many valued.
\end{theorem}
\begin{proof}

\noindent

\noindent
We show that a  formula  $\varphi\in{\tt Fm}$ is satisfied in all the elements of all the finite Kripke models if and only if it is mapped to the designated element under all the valuation functions:

\noindent
(1) If $\varphi$ is satisfied in any element of any finite  Kripke model, then for any valuation  $\boldsymbol\nu$ by Lemma~\ref{lem:aux}(1)  we have $\langle\!\langle\boldsymbol\nu(\varphi)
\rangle\!\rangle_n=\mathbf{1}_{K_n}$ for any $n\in\mathbb{N}$, so $\boldsymbol\nu(\varphi)=\boldsymbol\tau$.

\noindent
(2) If $\varphi$ is not satisfied in some element of some finite Kripke model, then for some $m\in\mathbb{N}$ there is a Kripke model $\mathcal{K}=(K_m,\succcurlyeq_m,\Vdash)$ such that $\Bbbk\nVdash\varphi$ for some $\Bbbk\in K_m$. So, by Lemma~\ref{lem:aux}(2) we have
$\langle\!\langle\boldsymbol\nu_m^{\Vdash}(\varphi)
 \rangle\!\rangle_m(\Bbbk)=0$,
thus $\boldsymbol\nu_m^{\Vdash}(\varphi)\neq
\boldsymbol\tau$.
\hfill \ding{113} \end{proof}

\section{Propositional Connectives inside    G\"odel-Dummet Logic}\label{sec:def}
 In   classical propositional logic (which is a two valued logic), all the connectives can be defined by (the so-called complete set of connectives) $\{\neg,\wedge\}, \{\neg,\vee\}$ or $\{\neg,\rightarrow\}$ only. In this last section  we will see that no propositional connective is definable from the others in IPL, and in GDL only the disjunction operation ($\vee$) can be defined by the conjunction ($\wedge$) and implication ($\rightarrow$) operations. Most of these facts are already known (they appear in e.g. \cite{ss-farsi} and \cite{svejd}).
Theorem~\ref{th:wedge} is from \cite{svejd} with a slightly  different proof; Theorem~\ref{th:impl} is from \cite{svejd} with the same proof.
All of our proofs  are Kripke model theoretic, as usual.

\begin{theorem}[$\boldsymbol\wedge$ Is Not Definable From the Others in GDL]\label{th:wedge}

\noindent
In G\"odel-Dummet Logic, the conjunction connective $(\wedge)$ is not definable from the other propositional connectives.
\end{theorem}


\begin{proof}

\noindent

\noindent
Consider the Kripke model $\mathcal{K}=\langle K,\succcurlyeq,\Vdash\rangle$ where
$K=\{a,b,c\}$,
$\succcurlyeq$ is the reflexive closure of $\{(a,b),(c,b)\}$, and
$\Vdash\;=\{(a,\mathfrak{p}),
(b,\mathfrak{p}),(b,\mathfrak{q}),(c,\mathfrak{q})\}$, for atoms $\mathfrak{p},\mathfrak{q}\in{\tt At}$.
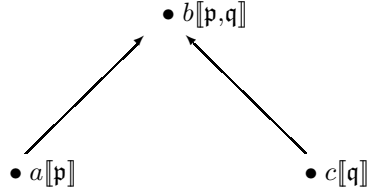
\begin{diagram}
& & \hspace{3em}\bullet\; b[\![\mathfrak{p},\!\mathfrak{q}]\!] & & \\
  & \ruTo &&\luTo& \\
\hspace{3em}\bullet\;a[\![\mathfrak{p}]\!] & & & & \hspace{0.75em}\bullet\; c[\![\mathfrak{q}]\!]
\end{diagram}
\noindent
We show that for all formulas  $\theta\in\mathcal{L}(\neg,\vee,\rightarrow,
\boldsymbol\top,
\mathfrak{p},\mathfrak{q})$ we have:

$(\ast)\qquad
b\Vdash\theta\Longrightarrow a\Vdash\theta\textrm{ or }c\Vdash\theta$.

\noindent
This will prove the desired conclusion, since  $b\Vdash\mathfrak{p}\!\wedge\!\mathfrak{q}$ but $a,c\nVdash\mathfrak{p}\!\wedge\!\mathfrak{q}$, and so  
$\mathfrak{p}\!\wedge\!\mathfrak{q}$ cannot belong to $\mathcal{L}(\neg,\vee,\rightarrow,
\boldsymbol\top,
\mathfrak{p},\mathfrak{q})$. We prove $(\ast)$ by induction on $\theta$. The cases of $\theta=\boldsymbol\top,
\mathfrak{p},\mathfrak{q}$ are trivial, and the induction step of $\neg\varphi$ follows from Remark~\ref{rem}, and the case of $\varphi\vee\psi$ is rather easy. So, only the non-trivial case of $\theta=\varphi\!\rightarrow\!\psi$ remains. Suppose that $(\ast)$ holds for $\varphi$ and $\psi$, and assume (for the sake of a contradiction) that $b\Vdash\varphi\!\rightarrow\!\psi$ but $a,c\nVdash\varphi\!\rightarrow\!\psi$.
So, $a\Vdash\varphi$ and $a\nVdash\psi$; and also $c\Vdash\varphi$ and $c\nVdash\psi$. Whence, by persistency, we should have also   $b\Vdash\varphi$, thus $b\Vdash\psi$. So, by the induction hypothesis $(\ast \textrm{ for } \theta\!=\!\psi)$ we should have either $a\Vdash\psi$ or $c\Vdash\psi$; a contradiction.
\hfill \ding{113} \end{proof}

\begin{theorem}[$\boldsymbol\rightarrow$ Is Not Definable From the Others in GDL]\label{th:impl}

\noindent
In G\"odel-Dummet Logic, the implication connective $(\rightarrow)$ is not definable from the other propositional connectives.
\end{theorem}

\begin{proof}

\noindent

\noindent
For  the Kripke model $\mathcal{K}=\langle K,\succcurlyeq,\Vdash\rangle$,  where
$K=\{a,b,c\}$,
$\succcurlyeq$ is the reflexive closure of $\{(a,b),(c,b)\}$, and
$\Vdash\;=\{(a,\mathfrak{p}),
(b,\mathfrak{p}),(b,\mathfrak{q})\}$, for $\mathfrak{p},\mathfrak{q}\in{\tt At}$,
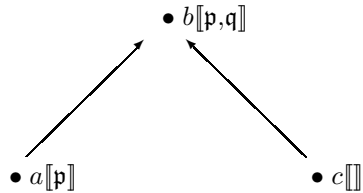
\begin{diagram}
& & \hspace{3em}\bullet\; b[\![\mathfrak{p},\!\mathfrak{q}]\!] & & \\
  & \ruTo &&\luTo& \\
\hspace{3em}\bullet\;a[\![\mathfrak{p}]\!] & & & & \hspace{0.75em}\bullet\; c[\![]\!]
\end{diagram}
\noindent
we show that for all the formulas $\theta\in\mathcal{L}(\neg,\vee,\wedge,
\boldsymbol\top,
\mathfrak{p},\mathfrak{q})$, the following holds:

$(\ast)\qquad b,c\Vdash\theta\Longrightarrow a\Vdash\theta$.

\noindent
This completes the proof since  $b,c\Vdash\mathfrak{p}\!\rightarrow\!\mathfrak{q}$ but $a\nVdash\mathfrak{p}\!\rightarrow\!\mathfrak{q}$ (by $a\Vdash\mathfrak{p},a\nVdash\mathfrak{q}$);  thus we have $(\mathfrak{p}\!\rightarrow\!\mathfrak{q})\not\in
\mathcal{L}(\neg,\vee,\wedge,
\boldsymbol\top,
\mathfrak{p},\mathfrak{q})$. The proof of $(\ast)$ is by induction on $\theta$; the only non-trivial cases to consider are $\theta=\varphi\!\vee\!\psi$ and $\theta=\varphi\!\wedge\!\psi$. Suppose that $(\ast)$ holds for $\varphi$ and $\psi$; and that $b,c\Vdash\varphi\!\vee\!\psi$. Then we have either $c\Vdash\varphi$ or $c\Vdash\psi$; by the persistency, the former implies $b\Vdash\varphi$ and the latter $b\Vdash\psi$. So, in either case by the induction hypothesis
we have $a\Vdash\varphi\!\vee\!\psi$. The case of $\theta=\varphi\!\wedge\!\psi$ is even simpler.
\hfill \ding{113} \end{proof}

The following has been known for a long time; see e.g. \cite{dummet}.

\begin{theorem}[$\boldsymbol\vee$ {\em Is}  Definable From $\boldsymbol\wedge,\boldsymbol\rightarrow$ in GDL]\label{th:vee}

\noindent
In G\"odel-Dummet Logic, the disjunction connective $(\vee)$ {\em is} definable from some other propositional connectives.
\end{theorem}
\begin{proof}

\noindent

\noindent
It is rather easy to see that
${\rm IPL}\Vdash (\mathfrak{p}\!\vee\!\mathfrak{q}) \longrightarrow [(\mathfrak{p}\!\rightarrow\!\mathfrak{q})
\!\rightarrow\!\mathfrak{q}]\!\wedge\!
[(\mathfrak{q}\!\rightarrow\!\mathfrak{p})
\!\rightarrow\!\mathfrak{p}]$. Now, we show that
${\rm GDL}\Vdash
[(\mathfrak{p}\!\rightarrow\!\mathfrak{q})
\!\rightarrow\!\mathfrak{q}]\!\wedge\!
[(\mathfrak{q}\!\rightarrow\!\mathfrak{p})
\!\rightarrow\!\mathfrak{p}]
\longrightarrow
 (\mathfrak{p}\!\vee\!\mathfrak{q})$ holds. Take an arbitrary {\em connected}  Kripke model $\mathcal{K}=\langle K,\succcurlyeq,\Vdash\rangle$, and suppose that for an arbitrary $a\in K$ we have $a\Vdash[(\mathfrak{p}\!\rightarrow\!\mathfrak{q})
\!\rightarrow\!\mathfrak{q}]\!\wedge\!
[(\mathfrak{q}\!\rightarrow\!\mathfrak{p})
\!\rightarrow\!\mathfrak{p}]$. We show that $a\Vdash\mathfrak{p}\!\vee\!\mathfrak{q}$. Assume not; then $a\nVdash\mathfrak{p},\!\mathfrak{q}$. Therefore, $a\nVdash(\mathfrak{p}\!\rightarrow\!\mathfrak{q})$ and $a\nVdash(\mathfrak{q}\!\rightarrow\!\mathfrak{p})$, by $a\Vdash[(\mathfrak{p}\!\rightarrow\!\mathfrak{q})
\!\rightarrow\!\mathfrak{q}]$ and $a\Vdash[(\mathfrak{q}\!\rightarrow\!\mathfrak{p})
\!\rightarrow\!\mathfrak{p}]$, respectively. So, there should exist some $b,c\in K$ with $b,c\!\succcurlyeq\!a$ such that $b\Vdash\mathfrak{p}$, $b\nVdash\mathfrak{q}$, $c\Vdash\mathfrak{q}$, and $c\nVdash\mathfrak{p}$.
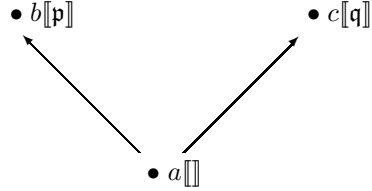
\begin{diagram}
\hspace{3em}\bullet\; b[\![\mathfrak{p}]\!] & &  & &
\hspace{1.5em}\bullet\;
c[\![\mathfrak{q}]\!] \\
  & \luTo &&  \ruTo& \\
 & & \hspace{1em}\bullet\;a[\![]\!] & & \end{diagram}
\noindent
By the connectivity of $\succcurlyeq$, we should have either $b\!\succcurlyeq\!c$ or $c\!\succcurlyeq\!b$. Both cases lead to a contradiction, by the persistency condition.  So, the following equivalence
$$(\mathfrak{p}\!\vee\!\mathfrak{q})\;\equiv\;
[(\mathfrak{p}\!\rightarrow\!\mathfrak{q})
\!\rightarrow\!\mathfrak{q}]\!\wedge\!
[(\mathfrak{q}\!\rightarrow\!\mathfrak{p})
\!\rightarrow\!\mathfrak{p}]$$
holds in GDL.
\hfill \ding{113} \end{proof}

The fact of the matter is that $(\mathfrak{p}\!\vee\!\mathfrak{q})\;\equiv\;
[(\mathfrak{p}\!\rightarrow\!\mathfrak{q})
\!\rightarrow\!\mathfrak{q}]\!\wedge\!
[(\mathfrak{q}\!\rightarrow\!\mathfrak{p})
\!\rightarrow\!\mathfrak{p}]$  is the only non-trivial equivalence relation between the propositional connectives in GDL. The first half of the following theorem was proved in \cite{ss-farsi}.

\begin{theorem}[In GDL $\boldsymbol\vee$ Is Not Definable Without Both $\boldsymbol\wedge,\boldsymbol\rightarrow$]
\label{th:notvee}

\noindent
In G\"odel-Dummet Logic, the disjunction connective $(\vee)$ is not definable from the other propositional connectives, unless both the conjunction and the implication connectives are present. In the other words, $\vee$ is  definable {\em neither} from the set  $\{\neg,\rightarrow,\boldsymbol\top\}$ {\em nor} from the set $\{\neg,\wedge,\boldsymbol\top\}$.
\end{theorem}
\begin{proof}

\noindent

\noindent
Take the Kripke model $\mathcal{K}=\langle K,\succcurlyeq,\Vdash\rangle$ with
$K=\{a,b,c,d\}$,
$\succcurlyeq\,=$ the reflexive  closure of $\{(a,b),(c,d)\}$, and
$\Vdash\;=\{(b,\mathfrak{p}),(d,\mathfrak{q})\}$, for $\mathfrak{p},\mathfrak{q}\in{\tt At}$.
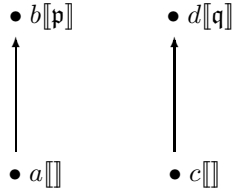
\begin{diagram}
\hspace{2em}\bullet\; b[\![\mathfrak{p}]\!] &  &  \hspace{2em}\bullet\;d[\![\mathfrak{q}]\!] \\
\uTo & & \uTo \\
\hspace{1.55em}\bullet\; a[\![]\!] &  & \hspace{1.55em}\bullet\; c[\![]\!]
\end{diagram}
\noindent
We show that for all $\theta\in\mathcal{L}(\neg,\rightarrow,
\boldsymbol\top,
\mathfrak{p},\mathfrak{q})$ we have

$(\ast)\qquad b,d\Vdash\theta\Longrightarrow a\Vdash\theta\textrm{ or }c\Vdash\theta$.

\noindent
Since  $b,d\Vdash\mathfrak{p}\!\vee\!\mathfrak{q}$ but $a,c\nVdash\mathfrak{p}\!\vee\!\mathfrak{q}$, then it follows that $\mathfrak{p}\!\vee\!\mathfrak{q}\not\in
\mathcal{L}(\neg,\rightarrow,
\boldsymbol\top,
\mathfrak{p},\mathfrak{q})$.

\noindent
Now, $(\ast)$ can be proved by induction on $\theta$; the only non-trivial case is $\theta=\varphi\!\rightarrow\!\psi$. If $(\ast)$ holds for $\varphi$ and $\psi$, then if $b,d\Vdash\varphi\!\rightarrow\!\psi$ but $a\nVdash\varphi\!\rightarrow\!\psi$ and $c\nVdash\varphi\!\rightarrow\!\psi$, then we should have $a\Vdash\varphi$ and $a\nVdash\psi$, and also $c\Vdash\varphi$ and $c\nVdash\psi$. So, by persistency, $b\Vdash\varphi$ and $d\Vdash\varphi$; thus $b\Vdash\psi$ and $d\Vdash\psi$. So, by the induction hypothesis $(\ast \textrm{ for } \theta\!=\!\psi)$  we should have either $a\Vdash\psi$ or $c\Vdash\psi$; a contradiction.

\noindent
Now, for proving $\mathfrak{p}\!\vee\!\mathfrak{q}\not\in
\mathcal{L}(\neg,\wedge,
\boldsymbol\top,
\mathfrak{p},\mathfrak{q})$, we show that
for all the formulas $\theta$ in $\mathcal{L}(\neg,\wedge,
\boldsymbol\top,
\mathfrak{p},\mathfrak{q})$ we have

$(\ddag)\qquad
b,d\Vdash\theta\Longrightarrow a,c\Vdash\theta$.

\noindent
Trivially, $(\ddag)$ holds for $\theta=\boldsymbol\top,
\mathfrak{p},\mathfrak{q}$; so by Remark~\ref{rem} it only suffices to show that  $(\ddag)$ holds for $\theta=\varphi\!\wedge\!\psi$, when it holds for $\varphi$ and $\psi$. Now, if $b,d\Vdash\varphi\!\wedge\!\psi$ then $b,d\Vdash\varphi$ and $b,d\Vdash\psi$; so
the induction hypothesis $(\ddag \textrm{ for } \theta\!=\!\varphi,\psi)$ implies that $a,c\Vdash\varphi$ and $a,c\Vdash\psi$, therefore $a,c\Vdash\varphi\!\wedge\!\psi$.
\hfill \ding{113} \end{proof}

We end the paper with a Kripke model theoretic proof of a known fact.

\begin{proposition}[No Connective Is Definable From the Others in IPL]\label{th:notvee}

\noindent

\noindent
In IPL, no propositional connective is definable from the others.
\end{proposition}
\begin{proof}

\noindent

\noindent
By Theorems~\ref{th:wedge} and ~\ref{th:impl}, $\wedge$ and $\rightarrow$ are not definable from the other connectives even in GDL. The statement   $\neg\mathfrak{p}\not\in
\mathcal{L}(\wedge,\vee,\rightarrow,\boldsymbol\top,
\mathfrak{p})$ can be easily verified by noting that all the operations on the righthand side are positive. So, all it remains is to show that we have $\mathfrak{p}\!\vee\!\mathfrak{q}\not\in
\mathcal{L}(\neg,\wedge,\rightarrow,
\boldsymbol\top,
\mathfrak{p},\mathfrak{q})$ in IPL (cf. Theorem~\ref{th:vee}). Consider the Kripke model $\mathcal{K}=\langle K,\succcurlyeq,\Vdash\rangle$ with
$K=\{a,b,c\}$,
$\succcurlyeq\,=$ the reflexive  closure of $\{(a,b),(a,c)\}$, and
$\Vdash\;=\{(b,\mathfrak{p}),(c,\mathfrak{q})\}$, for $\mathfrak{p},\mathfrak{q}\in{\tt At}$.
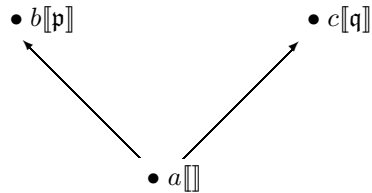
\begin{diagram}
\hspace{3em}\bullet\; b[\![\mathfrak{p}]\!] & &  & &
\hspace{1.5em}\bullet\;
c[\![\mathfrak{q}]\!] \\
  & \luTo &&  \ruTo& \\
 & & \hspace{1em}\bullet\;a[\![]\!] & & \end{diagram}
We show that for all formulas  $\theta\in\mathcal{L}(\neg,\wedge,\rightarrow,
\boldsymbol\top,
\mathfrak{p},\mathfrak{q})$ we have:

$(\ast)\qquad
b,c\Vdash\theta\Longrightarrow a\Vdash\theta$.

\noindent
This will prove the theorem, since  $b,c\Vdash\mathfrak{p}\!\vee\!\mathfrak{q}$ but $a\nVdash\mathfrak{p}\!\vee\!\mathfrak{q}$, and so  $\mathfrak{p}\!\vee\!\mathfrak{q}$ is not  in $\mathcal{L}(\neg,\wedge,\rightarrow,
\boldsymbol\top,
\mathfrak{p},\mathfrak{q})$ in IPL. Indeed, $(\ast)$ can be proved by induction on $\theta$; for which we consider the case of $\theta=\varphi\!\rightarrow\psi$ only. So, suppose that $(\ast)$ holds for $\varphi$ and $\psi$ and that $b,c\Vdash\varphi\!\rightarrow\!\psi$ but $a\nVdash\varphi\!\rightarrow\!\psi$. Then we should have $a\Vdash\varphi$ and $a\nVdash\psi$; but by persistency we should have that $b,c\Vdash\varphi$, and so $b,c\Vdash\psi$ holds. Now, the induction hypothesis $(\ast \textrm{ for } \theta\!=\!\psi)$ implies that $a\Vdash\psi$, a contradiction.
\hfill \ding{113} \end{proof}

%
%
%

\end{document}